\documentclass[reqno,12pt]{amsart}
\usepackage{amsmath,amssymb,color}
\usepackage{amsthm}
\usepackage{comment}
\usepackage{graphicx}

\newtheorem{theorem}{Theorem}

\newtheorem{proposition}{Proposition}
\newtheorem{remark}{Remark}

\def\re{\mathbb{R}}

\def\N{\mathbb{N}}

\def\Sp{\mathbb{S}}
\def\eps{\varepsilon}

\def\pd{\partial}
\def\ol{\overline}

\def\la{\lambda}

\def\disp{\displaystyle}
\def\({\left(}
\def\){\right)}

\def\pd{\partial}

\def\intO{\int_{\Omega}}

\def\|{\Vert}
\def\weakto{\rightharpoonup}

\begin{document}
\title[Quasilinear minimization problems]
{Some quasilinear minimization problems involving Hardy potentials}

\author{Futoshi Takahashi}

\address{
Department of Mathematics, Osaka Metropolitan University \\
3-3-138, Sumiyoshi-ku, Sugimoto-cho, Osaka, Japan \\
}

\email{futoshi@omu.ac.jp \\}

\author{Kota Urabe}

\address{
Department of Mathematics, Osaka Metropolitan University \\
3-3-138, Sumiyoshi-ku, Sugimoto-cho, Osaka, Japan \\
}

\email{parutt98@gmail.com \\}

\begin{abstract}
In this note, we consider several quasiliniear minimization problems involving various Hardy type potentials
for functions in $W^{1,p}(\Omega)$ with weighted mean zero.
We prove that the strict inequality for the infimum yields the existence of a minimizer.
\end{abstract}


\date{\today}

\dedicatory{}

\maketitle

%
%
\section{Introduction}

Let $\Omega \subset \re^N$, $N \ge 2$, be a smooth bounded domain with $0 \in \Omega$ and $1 < p < \infty$, $p \ne N$.
The well-known $W^{1,p}_0$-Hardy inequality states that if we define
\begin{equation*}
	H_p(\Omega) := \inf_{\substack{u \in W^{1,p}_0(\Omega \setminus \{ 0 \}) \\ u \neq 0}} \frac{\intO |\nabla u|^p dx}{\intO \frac{|u|^p}{|x|^p} dx},
\end{equation*}
then $H_p(\Omega) > 0$, $H_p(\Omega)$ does not depend on $\Omega$, in fact $H_p(\Omega) = H_p = \Big|\frac{N-p}{p}\Big|^p$, 
and the infimum is never attained in $W^{1,p}_0(\Omega \setminus \{ 0 \})$.
Note that $W^{1,p}_0(\Omega \setminus \{ 0 \}) =  W^{1,p}_0(\Omega)$ if and only if $1 \le p \le N$.

Let $1 < p < N$.
In stead of the space $W^{1,p}_0(\Omega)$, let us define a space
\begin{equation}
\label{X_p}
	X_p = \{ u \in W^{1,p}(\Omega) \ : \ \intO \frac{u}{|x|^p} dx = 0 \}
\end{equation}
and $W^{1,p}$-Hardy best constant for weighted mean zero functions
\begin{equation}
\label{Hbar}
	\bar{H}_p(\Omega) := \inf_{\substack{u \in X_p \\ u \neq 0}} \frac{\int_{\Omega} |\nabla u|^p dx}{\int_{\Omega} \frac{|u|^p}{|x|^p}dx}.
\end{equation}

Concerning $\bar{H}_p(\Omega)$ in \eqref{Hbar}, we obtain the following:
\begin{theorem}
\label{Thm1}
Let $1 < p < N$. Then we have $\bar{H}_p(\Omega) > 0$. 
Moreover, if we assume 
\[
	\bar{H}_p(\Omega) < H_p = \( \frac{N-p}{p} \)^p,
\]
then $\bar{H}_p(\Omega)$ is attained by a non-zero function $u \in X_p$.
\end{theorem}
Note that the attainability of $\bar{H}_p(\Omega)$ was first considered by Chabrowski-Peral-Ruf \cite{CPR} for $p=2$ and $N \ge 3$.

More generally, let $\Omega \subset \re^N$ be a bounded smooth domain,
$1 < p < \infty$ and $K \subset \Omega$ be a smooth submanifold with ${\rm codim} \, K = k$, $k=1,\cdots, N$. 
We assume that $K = \{ 0 \} \subset \Omega$ if $k = N$. 
When $k=1$, we exclude that $K = \pd\Omega$ since $K$ is assumed to satisfy $K \subset \Omega$.
Also note that $W^{1,p}_0(\Omega \setminus K) =  W^{1,p}_0(\Omega)$ if and only if $1 \le p \le k$.
Let $1 < p < \infty$, $p \ne k$, and put the assumption $(K)$ on the submanifold $K$:
\[
	-\Delta_p d(x, K)^{\frac{p-k}{p-1}} \ge 0 \quad \text{weakly in} \ W^{1,p}_0(\Omega \setminus K)
\]
where $d(x, K) = {\rm dist} (x, K)$ and $\Delta_p = {\rm div }(| \nabla(\cdot)|^{p-2} \nabla(\cdot))$ denotes the $p$-Laplacian.
Then the following $W^{1,p}_0$-geometric Hardy inequality involving the distance from a manifold is proved by
Barbatis-Filippas-Tertikas \cite{BFT}: Put
\begin{equation}
\label{J_p}
	J_{K, p}(\Omega) := \inf_{\substack{u \in W^{1,p}_0(\Omega \setminus K)\\ u \neq 0}}
	\frac{\int_{\Omega} |\nabla u|^p dx}{\int_{\Omega} \frac{|u|^p}{d(x, K)^p}dx},
\end{equation}
then $J_{K, p}(\Omega) > 0$, $J_{K, p}(\Omega)$ does not depend on $\Omega$ with the value $J_{K, p} = \Big| \frac{k-p}{p} \Big|^p$,
and $J_{K, p}$ is not attained in $W^{1,p}_0(\Omega \setminus K)$.

Now, we assume that $1 < p < k$.
As before, we define a weighted space
\begin{equation}
\label{Y_p}
	Y_p = \{ u \in W^{1,p}(\Omega) \ : \ \intO \frac{u}{d(x, K)^p} dx = 0 \}
\end{equation}
and
\begin{equation}
\label{Jbar}
	\bar{J}_{K,p}(\Omega) := \inf_{\substack{u \in Y_p \\ u \neq 0}} \frac{\int_{\Omega} |\nabla u|^p dx}{\int_{\Omega} \frac{|u|^p}{d(x, K)^p}dx}.
\end{equation}
Then we have the following:
\begin{theorem}
\label{Thm2}
Let $1 < p < k$ and assume that $K \subset \Omega$ satisfies the assumption (K).
Then we have $\bar{J}_{K, p}(\Omega) > 0$.
Moreover, if we assume 
\[
	\bar{J}_{K,p}(\Omega) < J_{K,p} = \(\frac{k-p}{p} \)^p,
\]
then $\bar{J}_{K,p}(\Omega)$ is attained by a non-zero function $u \in Y_p$.
\end{theorem}

\vspace{1em}
Also in the case of $p = N$ and $K = \{ 0 \} \subset \Omega$, we assume that $\Omega \subset\subset B$ where $B$ is the unit ball in $\re^N$.
In this situation, the critical Hardy inequality for functions in $W^{1,N}_0(\Omega)$ states that if we define 
\begin{equation}
\label{C_N}
	C_N(\Omega) := \inf_{\substack{u \in W^{1,N}_0(\Omega) \\ u \neq 0}} \frac{\intO |\nabla u|^N dx}{\intO \frac{|u|^N}{|x|^N (\log \frac{1}{|x|})^N} dx},
\end{equation}
then $C_N(\Omega) > 0$, $C_N(\Omega)$ does not depend on $\Omega$ with the value $C_N = \(\frac{N-1}{N}\)^N$,
and the infimum is again not attained in $W^{1,N}_0(\Omega)$.

As before, we define the space 
\begin{equation}
\label{X_N}
	X_N = \{ u \in W^{1,N}(\Omega) \ : \ \intO \frac{u}{|x|^N (\log \frac{1}{|x|})^N} dx = 0 \}
\end{equation}
and the $W^{1,N}$-Hardy best constant for weighted mean zero functions
\begin{equation}
\label{Cbar}
	\bar{C}_N(\Omega) := \inf_{\substack{u \in X_N \\ u \neq 0}} \frac{\int_{\Omega} |\nabla u|^N dx}{\int_{\Omega} \frac{|u|^N}{|x|^N (\log \frac{1}{|x|})^N} dx}.
\end{equation}
Then we have
\begin{theorem}
\label{Thm3}
Let $N \ge 2$.
Then we have $\bar{C}_N(\Omega) > 0$.
Moreover, if we assume 
\[
	\bar{C}_N(\Omega) < C_N = \(\frac{N-1}{N}\)^N,
\]
then $\bar{C}_N(\Omega)$ is attained by a non-zero function $u \in X_N$.
\end{theorem}
The attainability of $\bar{C}_N(\Omega)$ was first considered by Sano-Takahashi \cite{Sano-TF(JGA)} for $N=2$.

In many constrained variational problems, we have to treat weakly convergent sequences on appropriate function spaces,
say, $W^{1,p}(\Omega)$ or $W^{1,p}_0(\Omega)$ for $1 < p < \infty$.
One of the fundamental tools to treat such problems is the Brezis-Lieb Lemma \cite{Brezis-Lieb}, which states that
for $0 < p < \infty$ and $\{u_n\} \subset L^p(\Omega)$ satisfying
\begin{enumerate}
\item[(i)] $u_n(x) \to u(x)$ a.e. $x \in \Omega$,
\item[(ii)] $\disp{\sup_n \|u_n\|_{L^p(\Omega)}} < \infty$,
\end{enumerate}
the limit in the next exists and the following holds:
\begin{equation}
\label{BL}
	\lim_{n \to \infty} \left( \|u_n\|_{L^p(\Omega)}^p - \|u_n - u\|_{L^p(\Omega)}^p \right) = \| u \|_{L^p(\Omega)}^p.
\end{equation}
Note that if (i), (ii) holds, then we have 
\[
	u_n \weakto u \ \text{weakly in} \ L^p(\Omega) \quad (1 < p < \infty),
\]
see e.g, Hewitt-Stromberg's textbook \cite{HS} Theorem 13.44.
If $p=2$, then the Brezis-Lieb identity \eqref{BL} holds under the assumption $u_n \overset{w}{\weakto} u$ weakly in $L^2(\Omega)$ only.
Since $W^{1,p}(\Omega)$ is reflexive for $1 < p < \infty$,
we can subtract a weak convergent subsequence from a norm bounded sequence $\sup_{n \in \N} \| u_n \|_{W^{1,p}(\Omega)} < \infty$.
However, because of the lack of information of pointwise convergence of $\nabla u_n$, 
\begin{equation}
\label{BL_gradient}
	\intO |\nabla u_n|^p dx - \intO |\nabla (u_n - u)|^p dx = \intO |\nabla u|^p dx + o_n(1) 
\end{equation}
does not hold true for a weak convergent sequence $u_n \overset{w}{\weakto} u$ in $W^{1,p}(\Omega)$ if $p \ne 2$ in general,
which is shown by a simple one-dimensional example.
On the other hand, in the Hilbert space case $p = 2$, \eqref{BL_gradient} is easily verified for any weak convergent sequences $\{ u_n \} \subset W^{1,2}(\Omega)$.

In order to overcome this difficulty in the quasilinear minimization problems $(p \ne 2)$, 
many authors study sufficient conditions for the Brezis-Lieb type identity of the gradient \eqref{BL_gradient} holds true for weak convergent sequences, 
see \cite{Adimurthi-Tintarev}, \cite{BM}, \cite{VW}.
In this paper, we use a theorem proved by Valeriola-Willem \cite{VW} to obtain results above.
In \cite{VW}, the authors studied the quasilinear minimization problems involving the critical Sobolev exponent,
which is suffering from the non-compactness of the embedding, say, $W^{1,p}(\Omega) \hookrightarrow L^{p^*}(\Omega)$,
where $p^* = \frac{Np}{N-p}$ is the critical Sobolev exponent.
In this paper, we treat the problems of which difficulties arise from the non-compactness of the embedding
say, $W^{1,p}(\Omega) \hookrightarrow L^p(\Omega, \frac{dx}{|x|^p})$, where $L^p(\Omega, \frac{dx}{|x|^p})$ is the weighted Lebesgue space with the weight $|x|^{-p}$.

%
%
\section{Preliminaries}

In this section, we collect several facts needed for the proof of our results.

First, we prove the following Cherrier-type Hardy inequalities:
Let $\Omega \subset \re^N$ be a bounded smooth domain 
and let $K \subset \Omega$ be a smooth submanifold with ${\rm codim} \, K = k$, $k=1,\cdots, N$
satisfying the assumption $(K)$.
Then we have

\begin{proposition}
\label{Cherrier_J}
Let $1 < p < k$ and let $J_{K,p} = \(\frac{k-p}{p}\)^p$.
Then for any $\eps > 0$, there exists $C(\eps) > 0$ such that
\begin{equation}
\label{Eq:Cherrier_J}
	\( J_{K,p} - \eps \) \intO \frac{|u|^p}{d(x, K)^p} dx \le \intO |\nabla u|^p dx + C(\eps) \intO |u|^p dx
\end{equation}
holds for any $u \in W^{1,p}(\Omega) (= W^{1,p}(\Omega \setminus K))$.
\end{proposition}

\begin{proof}
Recall the following elementary inequality:
Let $p \ge 1$. Then for any $\eps > 0$, there exists $C(\eps) > 0$ such that for all $a, b \in \re^N$,
\begin{equation}
\label{elementary}
	|a + b|^p \le (1 + \eps) |a|^p + C(\eps) |b|^p
\end{equation}
holds.
Let $u \in W^{1,p}(\Omega)$ be given.
Since $K$ is a compact submanifold of codimension $k=1,2,\cdots, N$ contained in $\Omega$,
there exists open sets $K \subset\subset U \subset\subset \Omega$ 
and $\phi \in C_0^{\infty}(\Omega)$ such that $\phi \equiv 1$ on $U$, $0 \le \phi \le 1$. 
Apply the Hardy inequality \eqref{J_p} by \cite{BFT} for $\phi u \in W^{1,p}_0(\Omega)$, 
we obtain
\begin{equation}
\label{H-1}
	J_{K,p} \intO \frac{|\phi u|^p}{d(x, K)^p} dx \le \intO |\nabla (\phi u)|^p dx.
\end{equation}
By the elementary inequality \eqref{elementary}, we see
\begin{align*}
	|\nabla (\phi u)|^p = |(\nabla \phi) u + (\nabla u)\phi|^p \le (1 + \eps') |\phi|^p|\nabla u|^p + C(\eps') |\nabla \phi|^p |u|^p
\end{align*}
holds for any $\eps' > 0$.
Thus the right hand-side of \eqref{H-1} can be estimated from above as
\begin{equation}
\label{H-2}
	\intO |\nabla (\phi u)|^p dx \le (1 + \eps') \intO |\nabla u|^p dx + C(\eps') \(\max_{x \in \Omega} |\nabla \phi|^p\)  \intO |u|^p dx.
\end{equation}
On the other hand,
$1-\phi \equiv 0$ on $U$.
Thus $\frac{|(1-\phi)|^p}{d(x, K)^p} \in L^{\infty}(\Omega \setminus K)$ and again by the elementary inequality \eqref{elementary}, we obtain
\begin{align}
\label{H-3}
	\intO \frac{|u|^p}{d(x, K)^p} dx &= \intO \frac{|\phi u + (1-\phi) u|^p}{d(x, K)^p} dx \notag \\ 
	&\overset{\eqref{elementary}}{\le} (1+\eps'') \intO \frac{|\phi u|^p}{d(x, K)^p} dx + C(\eps'') \intO  \frac{|(1-\phi)|^p|u|^p}{d(x, K)^p} dx \notag \\ 
	&\le (1+\eps'') \intO \frac{|\phi u|^p}{d(x, K)^p} dx + C(\eps'') \max_{x \in \Omega} \frac{|(1-\phi)|^p}{d(x, K)^p} \intO  |u|^p dx.
\end{align}
Insert \eqref{H-2}, \eqref{H-3} into \eqref{H-1}, then
\begin{align*}
	\intO \frac{|u|^p}{d(x, K)^p} dx &\overset{\eqref{H-3}}{\le} (1 + \eps'') \intO \frac{|\phi u|^p}{d(x, K)^p} dx + C(\eps'') \max_{x \in \Omega} \frac{|(1-\phi)|^p}{d(x, K)^p} \intO |u|^p dx \\
	&\overset{\eqref{H-1}}{\le} (1+\eps'') J_{K,p}^{-1} \intO |\nabla (\phi u)|^p dx + C(\eps'') \max_{x \in \Omega} \frac{|(1-\phi)|^p}{|x|^p} \intO |u|^p dx \\
	&\overset{\eqref{H-2}}{\le} (1+\eps'') J_{K,p}^{-1} \left\{ (1+\eps') \intO |\nabla u|^p dx + C(\eps') \(\max_{x \in \Omega} |\nabla \phi|^p\)  \intO |u|^p dx \right\} \\
	&+ C(\eps'') \max_{x \in \Omega} \frac{|(1-\phi)|^p}{d(x, K)^p} \intO |u|^p dx \\
	&= J_{K,p}^{-1} (1+\eps'')(1+\eps') \intO |\nabla u|^p dx \\
	&+ \underbrace{\left\{ J_{K,p}^{-1} (1+\eps'') C(\eps') \(\max_{x \in \Omega} |\nabla \phi|^p\) + C(\eps'') \max_{x \in \Omega} \frac{|(1-\phi)|^p}{d(x, K)^p} \right\}}_{=: C(\eps',\eps'')} \intO  |u|^p dx \\
	&= J_{K,p}^{-1} (1+\eps'')(1+\eps') \intO |\nabla u|^p dx + C(\eps',\eps'') \intO |u|^p dx.
\end{align*}
Thus for any $\eps > 0$, if we take $\eps', \eps'' > 0$ sufficiently small so that
\[
	\( J_{K,p} - \eps \) J_{K,p}^{-1} (1+\eps')(1+\eps'') \le 1,
\]
then we have
\[
	\( J_{K,p} - \eps \) \intO \frac{|u|^p}{d(x, K)^p} dx \le \intO |\nabla u|^p dx + C(\eps) \intO |u|^p dx,
\]
here
\[
	C(\eps) = C(\eps', \eps'') \( J_{K,p} -\eps \).
\]
\end{proof}

Let $B$ denote the unit ball in $\re^N$, $N \ge 2$.
By a similar proof as above, we obtain the following $W^{1,N}(\Omega)$-Cherrier type Hardy inequality.
Since the proof is quite similar, we omit it here. 

\begin{proposition}
\label{Cherrier_N}
Let $\Omega \subset\subset B \subset \re^N$ be a domain containing $\{ 0 \}$ 
and put $C_N = \(\frac{N-1}{N}\)^N$.
Then for any $\eps > 0$, there exists $C(\eps) > 0$ such that
\begin{equation}
\label{Eq:Cherrier_N}
	\( C_N - \eps \) \intO \frac{|u|^N}{|x|^N (\log \frac{1}{|x|})^N} dx \le \intO |\nabla u|^N dx + C(\eps) \intO |u|^N dx
\end{equation}
holds for any $u \in W^{1,N}(\Omega)$.
\end{proposition}

Define 
\[
	T(s) =
	\begin{cases}
	s, & (|s| \le 1), \\[4pt]
	\dfrac{s}{|s|}, & (|s| > 1).
	\end{cases}
\]

\begin{proposition}(Valeriola--Willem \cite{VW})
\label{Prop:VW}
Let $\Omega \subset \re^N$ be a bounded domain and $p > 1$.
Assume $\{u_n\} \subset W^{1,p}(\Omega)$ satisfy 
\begin{enumerate}
\item[(a)] $u_n \overset{w}{\weakto} u$ in  $W^{1,p}(\Omega)$,
\item[(b)] 
\[
	\int_{\Omega} \( |\nabla u_n|^{p-2} \nabla u_n - |\nabla u|^{p-2} \nabla u \) \cdot \nabla T(u_n - u)\, dx
	\to 0 \quad (n \to \infty).
\]
\end{enumerate}
Then the following holds:
\begin{enumerate}
\item[(i)] There exists a subsequence $\{u_{n_k}\}$ such that
\[
	\nabla u_{n_k} \to \nabla u \quad \text{a.e. on } \Omega.
\]
\item[(ii)] (Brezis-Lieb type Lemma for the gradient)
\[
	\lim_{n \to \infty} \( \int_{\Omega} |\nabla u_n|^p dx - \int_{\Omega} |\nabla (u_n - u)|^p dx \) = \int_{\Omega} |\nabla u|^p dx.
\]
\item[(iii)] For $1 \le q < p$, we have 
\[
	u_n \to u \quad \text{in } W^{1,q}(\Omega).
\]
\end{enumerate}
\end{proposition}

We need a version of Ekeland variational principle which assures the existence of the ``minimizing" Palais-Smale sequence for constrained variational problems.
We refer the readers to N. Ghoussoub \cite{Ghoussoub} Theorem 3.2 (based on a $C^1$-deformation lemma, Lemma 3.7), 
M. Cuesta \cite{Cuesta} Theorem 2.1 (based on the original Ekeland's variational principle and which avoids the use of deformation arguments), 
M. Willem \cite{Willem} Theorem 8.5 (in this reference, $C^2$-regularity of the constraint manifold is assumed to obtain an appropriate tangent pseudo-gradient vector field in a deformation theorem, however, it can be bypassed),
or most directly, the classic paper by I. Ekeland \cite{Ekeland(JMAA)} Corollary 3.4, Theorem 3.1.

\begin{proposition}
\label{prop:Ekeland}
Let $X$ be a Banach space, $F \in C^1(X, \re)$, $G \in C^1(X, \re)$, and define $C^1$-manifold
\[
	V := \{ v \in X \, : \, G(v) = 1 \}
\]
Assume that $G'(v) \ne 0$ for any $v \in V$ and $F$ is bounded from below on $V$.
Take any $v \in V$ and $\eps > 0$.
If $v$ satisfies
\[
	F(v) \le \inf_{V} F + \eps,
\]
then there exists $u_{\eps} \in V$ such that
\[
	F(u) \le \inf_{V} F + \eps^2 \quad \text{and} \quad \min_{\la \in \re} \| F'(u_{\eps}) - \la G'(u_{\eps}) \|_{X^*} \le \eps
\]
holds true.
\end{proposition}

%
%
\section{Proof of theorems}

\vspace{1em}\noindent
{\it Proof of Theorem \ref{Thm1} and Theorem \ref{Thm2}}.
Theorem \ref{Thm1} is a special case of Theorem \ref{Thm2} ($K = \{ 0 \} \subset \Omega$, $d(x, K) = |x|$), 
so it is enough to prove Theorem \ref{Thm2}.
Recall \eqref{Jbar}:
\begin{align*}
	\bar{J}_{K,p}(\Omega) = \inf_{\substack{u \in Y_p \\ u \neq 0}} \frac{\int_{\Omega} |\nabla u|^p dx}{\int_{\Omega} \frac{|u|^p}{d(x, K)^p}dx}
\end{align*}
for $1 < p < k$, where $Y_p$ is as in \eqref{Y_p}.
%
%

Define $A: W^{1,p}(\Omega) \to \re$ as
\[
	A(u) := \intO \frac{u}{d(x, K)^p} \, dx, \quad u \in W^{1,p}(\Omega).
\]
Then $Y_p = \{ u \in W^{1,p}(\Omega) : A(u) = 0 \}$.
By Cherrier-type Hardy inequality \eqref{Eq:Cherrier_J} in Proposition \ref{Cherrier_J} and the fact that $\frac{1}{d(x, K)^p} \in L^1(\Omega)$ when $p < k$,
we have
\begin{align*}
	|A(u)| \le \(\intO \frac{|u|^p}{d(x, K)^{p}} \)^{1/p} \(\intO \frac{dx}{d(x, K)^p} \)^{1-1/p} \le C \(\intO |\nabla u|^p + |u|^p dx \)^{1/p}.
\end{align*}
Therefore $A$ is a bounded linear functional on $W^{1,p}(\Omega)$ and $Y_p$ is a closed subspace in $W^{1,p}(\Omega)$.

%
%
First, we prove that $\bar{J}_{K, p}(\Omega) > 0$. 
Assume the contrary that $\bar{J}_{K, p}(\Omega) = 0$. 
Then there exists a sequence $\{ u_n \} \subset W^{1,p}(\Omega)$ such that
\begin{align*}
	\begin{cases}
	&\intO |\nabla u_n|^p dx \to \bar{J}_{K, p}(\Omega) = 0, \\
	&\intO \frac{u_n}{d(x, K)^p} dx = 0 \quad (n \in \N), \\
	&\intO \frac{|u_n|^p}{d(x, K)^p} dx = 1 \quad (n \in \N).
	\end{cases}
\end{align*}
Since
\[
	\intO |u_n|^p dx \le (\max_{x \in \Omega} d(x, K)^p) \intO \frac{|u_n|^p}{d(x, K)^p} dx = O_n(1),
\]
we see that $\{ u_n \} \subset Y_p$ is a norm bounded sequence in $W^{1,p}(\Omega)$.
Then up to a subsequence, there exists $u \in W^{1,p}(\Omega)$ such that
\begin{align*}
	\begin{cases}
	&u_n \overset{w}{\weakto} u \quad \text{weakly in} \ W^{1,p}(\Omega), \\
	&u_n \overset{s}{\to} u \quad \text{in} \ L^p(\Omega), \\
	&u_n(x) \to u(x) \quad a.e. \ \Omega.
	\end{cases}
\end{align*}
Since $Y_p$ is weakly closed in $W^{1,p}(\Omega)$ and $u_n \in Y_p$, we see that 
\[
	\int_{\Omega} \frac{u}{d(x, K)^p} dx = 0.
\] 
Now, Banach-Steinhaus theorem implies that
\[
	\intO |\nabla u|^p dx \le \liminf_{n \to \infty} \intO |\nabla u_n|^p dx = 0,
\]
so $u$ is a constant on $\Omega$.
Also since $u \in Y_p$, we obtain $u \equiv 0$ and $u_n \to u \equiv 0$ in $L^p(\Omega)$.
Let us take an open set $U \subset \Omega$ such that $K \subset U$. 
Take $\phi \in C_0^{\infty}(\Omega)$, $0 \le \phi \le 1$, $\phi \equiv 1$ on $U$,
and consider $v_n = \phi u_n \in W^{1, p}_0(\Omega)$.
Then as in \eqref{H-2} in the proof of Proposition \ref{Cherrier_J}, we see 
\[
	\intO |\nabla v_n|^p dx \le (1 + \eps) \intO |\nabla u_n|^p dx + C \(\max_{x \in \Omega} |\nabla \phi|^p \)  \intO |u_n|^p dx = o_n(1).
\]
Also since $\frac{|\phi|^p}{d(x, K)^p} \in L^{\infty}(\Omega \setminus \bar{U})$,
wee see $\int_{\Omega \setminus \bar{U}} \frac{|\phi u_n|^p}{d(x, K)^p} dx = o_n(1)$.
Thus
\begin{align*}
	&\intO \frac{|v_n|^p}{d(x, K)^p} dx = \int_{\Omega \setminus \bar{U}} \frac{|\phi u_n|^p}{d(x, K)^p} dx + \int_{U} \frac{|u_n|^p}{d(x, K)^p} dx \\
	&= o_n(1) +  \intO \frac{|u_n|^p}{d(x, K)^p} dx - \int_{\Omega \setminus \bar{U}} \frac{|u_n|^p}{d(x, K)^p} dx \\
	&= 1 + o_n(1).
\end{align*}
Thus we see that $\intO \frac{|v_n|^p}{d(x, K)^p} dx \ge C > 0$ for some $C > 0$.
Then by testing \eqref{J_p} by $v_n \in W^{1,p}_0(\Omega)$, we obtain a contradiction that 
\[
	0 < \(\frac{k-p}{p}\)^p \le \frac{\intO |\nabla v_n|^p dx}{\intO \frac{|v_n|^p}{d(x, K)^p} dx} = o_n(1).
\]
Thus $\bar{J}_{K, p}(\Omega) > 0$.

Next, define $C^1$-functionals $F, G: Y_p \to \re$ as
\begin{align*}
	F(u) := \int_{\Omega} |\nabla u|^p dx,  \quad G(u) := \int_\Omega \frac{|u|^p}{d(x, K)^p} dx.
\end{align*}
Then by Ekeland variational principle Proposition \ref{prop:Ekeland}, there exists $\{u_n\} \subset Y_p$ and $\la_n \in \re$ such that 
\begin{align*}
	&F(u_n) \to \bar{J}_{K,p}(\Omega), \quad G(u_n) = 1 \quad (n \in \mathbb{N}), \\
	&h_n := \frac{1}{p} \( F'(u_n) - \la_n G'(u_n) \) \to 0 \quad \text{in } (Y_p)^{*}.
\end{align*}
By testing the second equation by $u_n$, it is easy to check that $\la_n = \bar{J}_{K,p}(\Omega) + o_n(1)$ as $n \to \infty$.
Define
\[
	d_n = \( \intO \frac{1}{d(x, K)^p} dx \)^{-1} \intO \frac{T(u_n-u)}{d(x, K)^p} dx, \quad (n \in \N).
\]
Then we see $\intO \frac{T(u_n-u) -d_n}{d(x, K)^p} dx = 0$ and $T(u_n - u) - d_n \in Y_p$.
Also since 
\[
	u_n(x) \to u(x) \quad a.e. \ \Omega, \quad \frac{|T(u_n-u)|}{d(x, K)^p} \le \frac{1}{d(x, K)^p} \in L^1(\Omega),
\] 
Lebesgue's dominated convergence theorem implies that $d_n = o_n(1)$.
We see
\begin{align*}
	\langle h_n, T(u_n - u) -d_n \rangle &= \int_{\Omega} |\nabla u_n|^{p-2} \nabla u_n \cdot \nabla T(u_n - u) dx \\
	&- (\bar{J}_{K,p}(\Omega) + o_n(1)) \int_{\Omega} \frac{|u_n|^{p-2} u_n \, (T(u_n - u) - d_n)}{{{d(x, K)}^p}} dx,
\end{align*}
which implies
\begin{align*}
	&\int_{\Omega} \( |\nabla u_n|^{p-2} \nabla u_n - |\nabla u|^{p-2} \nabla u \) \cdot \nabla T(u_n - u) dx \\
	&= \langle h_n, T(u_n - u) - d_n \rangle \quad \cdots (A) \\
	&+ (\bar{J}_{p,K}(\Omega) + o_n(1)) \intO \frac{|u_n|^{p-2} u_n \, T(u_n - u)}{{{d(x, K)}^p}} \, dx \quad \cdots (B) \\
	&- (\bar{J}_{p,K}(\Omega) + o_n(1)) d_n \intO \frac{|u_n|^{p-2} u_n}{{{d(x, K)}^p}} \, dx \quad \cdots (C) \\
	&- \int_{\Omega} |\nabla u|^{p-2} \nabla u \cdot \nabla T(u_n - u) dx. \quad \cdots (D)
\end{align*}
In the following, we estimate (A), (B), (C), (D) respectively.

\vspace{1em}\noindent
{\it Estimate of (A)}: Since $h_n = o_n(1)$ in $(Y_p)^*$ strongly and $\sup_{n \in \N} \| T(u_n - u) - d_n \|_{W^{1,p}(\Omega)} = O_n(1)$, 
we have
\[
	(A) = \langle h_n, T(u_n - u) - d_n \rangle = o_n(1) \quad (n \to \infty).
\]

\vspace{1em}\noindent
{\it Estimate of (B)}:
\begin{align*}
	\left| \int_{\Omega} \frac{|u_n|^{p-2} u_n \, T(u_n - u)}{{{d(x, K)}^p}} dx \right| 
	\le \(\intO \frac{|u_n|^p}{d(x, K)^p} dx\)^{\frac{p-1}{p}} \(\intO \frac{|T(u_n-u)|^p}{d(x, K)^p} dx\)^{\frac{1}{p}}
\end{align*}
and
\[
	 \frac{|T(u_n-u)|^p}{d(x, K)^p} \le \frac{1}{d(x, K)^p} \in L^1(\Omega), \quad \frac{|T(u_n-u)|^p}{d(x, K)^p} \to 0 \quad a.e. \ \text{in} \ \Omega.
\]
Thus $(B) = o_n(1)$ by Lebesgue's dominated convergence theorem.

\vspace{1em}\noindent
{\it Estimate of (C)}:
Since
\begin{align*}
	\left| \int_{\Omega} \frac{|u_n|^{p-2} u_n)}{{{d(x, K)}^p}} dx \right| 
	\le \(\intO \frac{|u_n|^p}{d(x, K)^p} dx\)^{\frac{p-1}{p}} \(\intO \frac{1}{d(x, K)^p} dx\)^{\frac{1}{p}} = O_n(1)
\end{align*}
and $d_n = o_n(1)$, we have $(C) = o_n(1)$ as $n \to \infty$.

\vspace{1em}\noindent
{\it Estimate of (D)}:
Let
\[
	f: L^p(\Omega, \re^N) \to \re, \quad f(\vec{v}) = \intO |\nabla u|^{p-2} \nabla u \cdot \vec{v} \, dx.
\]
We can easily check that $f \in (L^p(\Omega, \re^N))^*$ and $\nabla T(u_n - u) \overset{w}{\weakto} 0$ weakly in $L^p(\Omega, \re^N)$.
Thus $(C) = f(\nabla T(u_n-u)) = o_n(1)$ by the definition of weak convergence.

In conclusion, the assumption (b) in Valeriola-Willem's Proposition \ref{Prop:VW}: 
\[
	\int_{\Omega} \( |\nabla u_n|^{p-2} \nabla u_n - |\nabla u|^{p-2} \nabla u \) \cdot \nabla T(u_n - u)\, dx
	\to 0 \quad (n \to \infty)
\]
holds for a minimizing Palais-Smale sequence $\{ u_n \}$.
Thus by Proposition \ref{Prop:VW} we have the Brezis-Lieb type identity \eqref{BL_gradient} for $\{ \nabla u_n \}$ and $\nabla u$:
\[
	\lim_{n \to \infty} \( \int_{\Omega} |\nabla u_n|^p dx - \int_{\Omega} |\nabla (u_n - u)|^p dx \) = \int_{\Omega} |\nabla u|^p dx.
\]
Now, we claim
\[
	\intO \frac{|u|^p}{d(x, K)^p} dx = 1.
\]
If we show this claim, since we already know that $u \in Y_p$, we have
\[
	\bar{H}_p(\Omega) \le \intO |\nabla u|^p dx \le \liminf_{n \to \infty} \intO |\nabla u_n|^p dx = \bar{H}_p(\Omega),
\]
which implies that $u \in Y_p$ is a minimizer of $\bar{H}_p(\Omega)$.

{\bf (Step 1)} $u \not\equiv 0$.

By contradiction, we assume $u \equiv 0$.
Rellich's theorem implies $u_n \overset{s}{\to} u \equiv 0$ in $L^p(\Omega)$ (up to a subsequence).
Then by Cherrier type $W^{1,p}(\Omega)$-Hardy inequality \eqref{Eq:Cherrier_J} in Proposition \ref{Cherrier_J},
we see
\begin{align*}
	\(\(\frac{k-p}{p}\)^p - \eps \) \underbrace{\intO \frac{|u_n|^p}{d(x, K)^p}dx}_{=1} \le
	\underbrace{\int_O |\nabla u_n|^p dx}_{= \bar{J}_{K,p}(\Omega) + o_n(1)} + C(\eps) \underbrace{\intO |u_n|^p dx}_{= o_n(1)}
\end{align*}
Letting $n \to \infty$ and then $\eps \to 0$, then we have $\(\frac{k-p}{p}\)^p \le \bar{J}_{K,p}(\Omega)$. 
This contradicts to the assumption $\bar{J}_{K,p}(\Omega) < \(\frac{k-p}{p}\)^p$.

{\bf (Step 2)} $\intO \frac{|u|^p}{d(x, K)^p} dx = 1$. 

Assume the contrary that $\intO \frac{|u|^p}{d(x, K)^p} dx < 1$. 
Then
\begin{align*}
	&0 < 1 - \intO \frac{|u|^p}{d(x, K)^p} dx = \intO \frac{|u_n|^p}{d(x, K)^p} dx - \intO \frac{|u|^p}{d(x, K)^p} dx \\
	&=  \intO \frac{|u_n - u|^p}{d(x, K)^p} dx + o_n(1) \\
	&\le (\(\frac{k-p}{p}\)^p - \eps)^{-1} \( \intO |\nabla (u_n - u)|^p dx + C(\eps) \intO |u_n-u|^p dx \) \\
	&= (\(\frac{k-p}{p}\)^p - \eps)^{-1} 
\( \intO |\nabla u_n|^p dx - \intO |\nabla u|^p dx + o_n(1) \) \\
	&\le (\(\frac{k-p}{p}\)^p - \eps)^{-1} \( \bar{J}_{K,p}(\Omega) + o_n(1) - \bar{J}_{K,p}(\Omega) \intO \frac{|u|^p}{|x|^p} dx \) \\
	&= (\(\frac{k-p}{p}\)^p - \eps)^{-1} \bar{J}_{K,p}(\Omega) \( 1 - \intO \frac{|u|^p}{d(x, K)^p} dx + o_n(1) \). 
\end{align*}
Here,
in the 2nd line, we have used Brezis-Lieb identity for $\{ \frac{|u_n(x)|}{d(x, K)} \}_{n \in \N}$,
in the 3rd line, we have used Cherrier-type inequality \eqref{Eq:Cherrier_J} in \eqref{Cherrier_J},
in the 4th line, we have used Brezis-Lieb identity for the gradient \eqref{BL_gradient},
in the 5th line, we have used the definition of $\bar{J}_{K,p}(\Omega)$.
Taking $n \to \infty$ first and then $\eps \to 0$, we get a contradiction $\(\frac{k-p}{p}\)^p \le \bar{J}_{K,p}(\Omega)$.

The proof of Theorem \ref{Thm1} is completed.
\qed

\vspace{1em}\noindent
{\it Proof of Theorem \ref{Thm3}}.
Recall the minimization problem \eqref{X_N}, \eqref{Cbar}:
\begin{align*}
	&X_N = \{ u \in W^{1,N}(\Omega) \ : \ \intO \frac{u}{|x|^N (\log \frac{1}{|x|})^N} dx = 0 \}, \\
	&\bar{C}_N(\Omega) = \inf_{\substack{u \in X_N \\ u \neq 0}} \frac{\int_{\Omega} |\nabla u|^N dx}{\int_{\Omega} \frac{|u|^N}{|x|^N (\log \frac{1}{|x|})^N} dx}.
\end{align*}
Put $A: W^{1,N}(\Omega) \to \re$ as
\[
	A(u) := \intO \frac{u}{|x|^N (\log \frac{1}{|x|})^N} \, dx, \quad u \in W^{1,N}(\Omega).
\]
Then $X_N = \{ u \in W^{1,N}(\Omega) \ : \ A(u) = 0 \}$.
Since
\begin{align*}
	|A(u)| &\le \(\intO \frac{|u|^N}{|x|^N (\log \frac{1}{|x|})^N} dx \)^{1/N} \(\intO \frac{dx}{|x|^N (\log \frac{1}{|x|})^N} \)^{1-1/N} \\
	&\le C \(\intO |\nabla u|^N + |u|^N dx \)^{1/N}
\end{align*}
by Cherrier-type Hardy inequality \eqref{Eq:Cherrier_N} in Proposition \ref{Cherrier_N}, $A$ is a bounded linear functional on $W^{1,N}(\Omega)$.
Thus $X_N$ is a closed subspace in $W^{1,N}(\Omega)$.

First, we prove that $\bar{C}_N(\Omega) > 0$. 
Assume the contrary that $\bar{C}_N(\Omega) = 0$. 
Then there exists a sequence $\{ u_n \} \subset W^{1,N}(\Omega)$ such that
\begin{align*}
	\begin{cases}
	&\intO |\nabla u_n|^N dx \to \bar{C}_N(\Omega) = 0, \\
	&\intO \frac{u_n}{|x|^N (\log \frac{1}{|x|})^N} dx = 0 \quad (n \in \N), \\
	&\intO \frac{|u_n|^N}{|x|^N (\log \frac{1}{|x|})^N} dx = 1 \quad (n \in \N).
	\end{cases}
\end{align*}
Since
\[
	\intO |\nabla u|^N dx \le \liminf_{n \to \infty} \intO |\nabla u_n|^N dx = 0,
\]
$u$ is a constant on $\Omega$.
Also since $u \in X_N$ by the weak closedness of $X_N$, we obtain $u \equiv 0$.
Take $\phi \in C_0^{\infty}(\Omega)$ such that $0 \le \phi \le 1$, $\phi \equiv 1$ in a neighborhood of $0 \in \Omega$ and consider $v_n = \phi u_n$.
Then as in the proof of Theorem \ref{Thm2}, we easily  see that $\intO |\nabla v_n|^N dx = o_n(1)$ and $\intO \frac{|v_n|^N}{|x|^N (\log \frac{1}{|x|})^N} dx \ge C > 0$ for some $C > 0$.
Then by testing \eqref{C_N} by $v_n \in W^{1,N}_0(\Omega)$, we obtain a contradiction that 
\[
	0 < \(\frac{N-1}{N}\)^N \le \frac{\intO |\nabla v_n|^N dx}{\intO \frac{|v_n|^p}{|x|^N (\log \frac{1}{|x|})^N} dx} = o_n(1).
\]
Thus $\bar{C}_N(\Omega) > 0$.

Define
\[
	d_n = \( \intO \frac{1}{|x|^N (\log \frac{1}{|x|})^N} dx \)^{-1} \intO \frac{T(u_n-u)}{|x|^N (\log \frac{1}{|x|})^N} dx, \quad (n \in \N).
\]
Then we see $\intO \frac{T(u_n-u) -d_n}{|x|^N (\log \frac{1}{|x|})^N} dx = 0$ and $T(u_n - u) - d_n \in X_N$.
Also since 
\[
	u_n(x) \to u(x) \quad a.e. \ \Omega, \quad \frac{|T(u_n-u)|}{|x|^N (\log \frac{1}{|x|})^N} \le \frac{1}{|x|^N (\log \frac{1}{|x|})^N} \in L^1(\Omega),
\] 
Lebesge's dominated convergence theorem implies that $d_n = o_n(1)$.
Define $F, G \in C^1(X_N, \re)$ as
\begin{align*}
	F(u) := \int_{\Omega} |\nabla u|^N dx, \quad G(u) := \int_\Omega \frac{|u|^N}{|x|^N (\log \frac{1}{|x|})^N} dx.
\end{align*}
By Ekeland variational principle Proposition \ref{prop:Ekeland}, there exists $\{u_n\} \subset X_N$ such that 
\begin{align*}
	&F(u_n) \to \bar{C}_N(\Omega), \quad G(u_n) = 1 \quad (n \in \mathbb{N}), \\
	&h_n := \frac{1}{N} \( F'(u_n) - (\bar{C}_N(\Omega) + o_n(1)) G'(u_n) \) \to 0 \quad \text{in } (X_N)^{*}.
\end{align*}
Since
\[
	\intO |u_n|^N dx \le \max_{x \in \Omega} (|x|^N (\log \frac{1}{|x|})^N) \intO \frac{|u_n|^N}{|x|^N (\log \frac{1}{|x|})^N} dx = O_n(1),
\]
we see that $\{ u_n \} \subset X_N$ is $W^{1,N}$-norm bounded. 
Thus up to a subsequence, there exists $u \in W^{1, N}(\Omega)$ such that
\begin{align*}
	\begin{cases}
	&u_n \overset{w}{\weakto} u \quad \text{weakly in} \ W^{1, N}(\Omega), \\
	&u_n \overset{s}{\to} u \quad \text{in} \ L^N(\Omega), \\
	&u_n(x) \to u(x) \quad a.e. \ \Omega.
	\end{cases}
\end{align*}
Since $X_N$ is weakly closed in $W^{1,N}(\Omega)$, we see that 
\[
	\int_{\Omega} \frac{u}{|x|^N (\log \frac{1}{|x|})^N} dx = 0
\] 
and by Fatou's lemma, 
\[
	0 \le \int_{\Omega} \frac{|u|^N}{|x|^N (\log \frac{1}{|x|})^N} dx \le \liminf_{n \to \infty} \int_{\Omega} \frac{|u_n|^N}{|x|^N (\log \frac{1}{|x|})^N}  dx = 1.
\]
We see
\begin{align*}
	\langle h_n, T(u_n - u) - d_n \rangle &= \int_{\Omega} |\nabla u_n|^{N-2} \nabla u_n \cdot \nabla T(u_n - u) dx \\
	&- (\bar{C}_N(\Omega) + o_n(1)) \intO \frac{|u_n|^{N-2} u_n \, T(u_n - u)}{{{|x|}^N (\log \frac{1}{|x|})^N}} dx,
\end{align*}
which implies
\begin{align*}
	&\intO \( |\nabla u_n|^{N-2} \nabla u_n - |\nabla u|^{N-2} \nabla u \) \cdot \nabla T(u_n - u) dx \\
	&= \langle h_n, T(u_n - u) - d_n \rangle \quad \cdots (A) \\
	&+ (\bar{C}_N(\Omega) + o_n(1)) \intO \frac{|u_n|^{N-2} u_n \, T(u_n - u)}{{{|x|}^N (\log \frac{1}{|x|})^N}} \, dx \quad \cdots (B) \\
	&- (\bar{C}_N(\Omega) + o_n(1)) d_n \intO \frac{|u_n|^{N-2} u_n}{{{|x|}^N (\log \frac{1}{|x|})^N}} \, dx \quad \cdots (C) \\
	&- \intO |\nabla u|^{N-2} \nabla u \cdot \nabla T(u_n - u) dx. \quad \cdots (D)
\end{align*}
In the following, we estimate (A), (B), (C), (D) respectively.

\vspace{1em}\noindent
{\it Estimate of (A)}:
Since $h_n = o_n(1)$ in $(X_N)^*$ and $\sup_{n \in \N} \| T(u_n-u) -d_n \|_{W^{1,N}(\Omega)} = O_n(1)$,
we see
\[
	(A) = \langle h_n, T(u_n - u) -d_n \rangle = o_n(1) \quad (n \to \infty).
\]

\vspace{1em}\noindent
{\it Estimate of (B)}:
\begin{align*}
	\left| \intO \frac{|u_n|^{N-2} u_n \, T(u_n - u)}{{{|x|}^N (\log \frac{1}{|x|})^N}} dx \right| 
	\le \(\intO \frac{|u_n|^N}{|x|^N (\log \frac{1}{|x|})^N} dx\)^{\frac{N-1}{N}} \(\intO \frac{|T(u_n-u)|^N}{|x|^N (\log \frac{1}{|x|})^N} dx\)^{\frac{1}{N}}
\end{align*}
and
\[
	 \frac{|T(u_n-u)|^N}{|x|^N (\log \frac{1}{|x|})^N} \le \frac{1}{|x|^N (\log \frac{1}{|x|})^N} \in L^1(\Omega), \quad \frac{|T(u_n-u)|^N}{|x|^N (\log \frac{1}{|x|})^N} \to 0 \quad a.e. \ \text{in} \ \Omega.
\]
Thus $(B) = o_n(1)$ by Lebesgue's dominated convergence theorem.

\vspace{1em}\noindent
{\it Estimate of (C)}:
Since
\begin{align*}
	\left| \intO \frac{|u_n|^{N-2} u_n}{{{|x|}^N (\log \frac{1}{|x|})^N}} dx \right| 
	\le \(\intO \frac{|u_n|^N}{|x|^N (\log \frac{1}{|x|})^N} dx\)^{\frac{N-1}{N}} \(\intO \frac{1}{|x|^N (\log \frac{1}{|x|})^N} dx\)^{\frac{1}{N}},
\end{align*}
and $d_n = o_n(1)$, we have $(C) = o_n(1)$.

\vspace{1em}\noindent
{\it Estimate of (D)}:
Let
\[
	f: L^N(\Omega, \re^N) \to \re, \quad f(\vec{v}) = \intO |\nabla u|^{N-2} \nabla u \cdot \vec{v} \, dx.
\]
We can easily check that $f \in (L^N(\Omega, \re^N))^*$ and $\nabla T(u_n - u) \overset{w}{\weakto} 0$ weakly in $L^N(\Omega, \re^N)$.
Thus $(C) = f(\nabla T(u_n-u)) = o_n(1)$ by the definition of weak convergence.

In conclusion, the assumption (2) in Valeriola-Willem's Proposition \ref{Prop:VW} 
\[
	\int_{\Omega} \( |\nabla u_n|^{N-2} \nabla u_n - |\nabla u|^{N-2} \nabla u \) \cdot \nabla T(u_n - u)\, dx
	\to 0 \quad (n \to \infty)
\]
holds for a minimizing Palais-Smale sequence $\{ u_n \}$.
Thus by Proposition \ref{Prop:VW} we have Brezis-Lieb type identiy \eqref{BL_gradient} when $p = N$:
\[
	\lim_{n \to \infty} \( \int_{\Omega} |\nabla u_n|^N dx - \int_{\Omega} |\nabla (u_n - u)|^N dx \) = \int_{\Omega} |\nabla u|^N dx.
\]

Now, we claim that
\begin{align*}
	\int_{\Omega} \frac{|u|^N}{|x|^N (\log \frac{1}{|x|})^N} dx =1.
\end{align*}
If we show this claim, since we already know that $u \in X_N$,
we have
\[
	\bar{C}_N(\Omega) \le \intO |\nabla u|^N dx \le \liminf_{n \to \infty} \intO |\nabla u_n|^N dx = \bar{C}_N(\Omega),
\]
which implies that $u \in X_N$ is a minimizer of $\bar{C}_N(\Omega)$.

{\bf (Step 1)} $u \not\equiv 0$.

Assume by contradiction that $u \equiv 0$.
Then by the compactness of the embedding $W^{1,N} (\Omega) \hookrightarrow L^N (\Omega)$, 
we see $\intO |u_n|^N dx \to 0$.
Thus by the critical Hardy inequality for $W^{1,N}(\Omega)$, we have
\begin{align*}
	1 = \intO \frac{|u_n|^N}{|x|^N (\log \frac{1}{|x|})^N} \, dx
	&\le (C_N - \eps)^{-1} \int_{\Omega} |\nabla u_n|^N dx + o_n(1) \\
	&\le (C_N - \eps)^{-1} (\bar{C}_N(\Omega) + o_n(1)) + o_n(1).
\end{align*}
Letting $m \to \infty$ and then $\eps \to 0$, we have $1 \le C_N \bar{C}_N(\Omega)$, which contradicts the assumption $\bar{C}_N(\Omega) < C_N$.
Therefore, $u \not\equiv 0$. 

{\bf (Step 2)} $\int_{\Omega} \frac{|u|^N}{|x|^N (\log \frac{1}{|x|})^N} dx =1$.

Let us assume by contradiction that the claim does not hold and $\int_\Omega \frac{|u|^N}{|x|^N ( \log \frac{1}{|x|} )^N} \, dx < 1$. 
Since $u_n -u \rightharpoonup 0$ in $W^{1,N}(\Omega)$, 
Rellich's theorem implies $\int_{\Omega} |u_n - u|^N \, dx = o_n(1)$,
and by the Cherrier type inequaliity \eqref{Eq:Cherrier_N} in Proposition \ref{Cherrier_N} for $W^{1,N}(\Omega)$-functions again, 
we have
\begin{align*}
	0 &< 1- \intO \frac{|u|^N}{|x|^N \( \log \frac{1}{|x|} \)^N} \, dx 
	= \intO \frac{|u_n|^N}{|x|^N \( \log \frac{1}{|x|} \)^N} \, dx - \intO \frac{|u|^N}{|x|^N \( \log \frac{1}{|x|} \)^N} \, dx \\
	&= \int_\Omega \frac{|u_n -u|^N}{|x|^N \( \log \frac{1}{|x|} \)^N} \, dx + o_n(1) \\
	&\le (C_N - \eps)^{-1} \( \intO |\nabla (u_n -u) |^N \,dx + C \intO |u_n - u|^N \, dx \) + o_n(1) \\
	&\le (C_N - \eps)^{-1} \( \intO |\nabla u_n|^N \,dx - \intO |\nabla u |^N \,dx \) + o_n(1) \\
	&\le (C_N - \eps)^{-1} \( \bar{C}_N(\Omega) - \bar{C}_N(\Omega) \intO \frac{|u|^N}{|x|^N (\log \frac{1}{|x|})^N} \,dx \) + o_n(1) \\
	&= (C_N - \eps)^{-1} \bar{C}_N(\Omega) \(1 - \intO \frac{|u|^N}{|x|^N (\log \frac{1}{|x|})^N} \,dx + o_n(1) \)
\end{align*}
as $n \to \infty$. 
Here,
in the 2nd line, we have used Brezis-Lieb identity for $\{ \frac{|u_n(x)|}{|x| (\log \frac{1}{|x|})} \}_{n \in \N}$,
in the 3rd line, we have used Cherrier-type inequality \eqref{Eq:Cherrier_N} in Proposition \ref{Cherrier_N},
in the 4th line, we have used Brezis-Lieb identity for the gradient \eqref{BL_gradient}.
Also in the 5th line, since $\intO \frac{u}{|x|^N \( \log \frac{1}{|x|} \)^N} \, dx =0$, 
we have $\bar{C}_N(\Omega) \intO \frac{|u|^N}{|x|^N \( \log \frac{1}{|x|} \)^N} \, dx \le \intO |\nabla u|^N \, dx$
by the definition of $\bar{C}_N(\Omega)$. 
Therefore letting $n \to \infty$, we have
\begin{align*}
	1- \int_\Omega \frac{|u|^N}{|x|^N \( \log \frac{1}{|x|} \)^N} \, dx \le (C_N - \eps)^{-1} \bar{C}_N(\Omega) \(1- \intO \frac{|u|^N}{|x|^N \( \log \frac{1}{|x|} \)^N} \, dx \),
\end{align*}
which implies that $1 \le (C_N - \eps)^{-1} \bar{C}_N(\Omega)$. 
Again, letting $\eps \to 0$, we have a contradiction to the assumption that $\bar{C}_N(\Omega) < C_N$.
Therefore, we have the claim.

The proof is now completed.
\qed

\section{Verification of assumptions in Theorems}

Concerning the assumption of Theorem \ref{Thm1}:
\[
	\bar{H}_p(\Omega) < \(\frac{N-p}{p} \)^p \quad (1 < p < N),
\]
where $\bar{H}_p(\Omega)$ is defined in \eqref{Hbar},
first we assume $p \ge 2$ and let $\Omega = B_R$ be a ball with center $0$ and radius $R > 0$. 
For $x = (x_1, \cdots, x_N)$, put $u_i(x) = x_i$. 
Then we have $u_i \in X_p$ by symmetry where $X_p$ is as in \eqref{X_p} and
\begin{align*}
	&|\nabla u_i(x)| = |e_i| = 1, \quad \int_{B_R} |\nabla u_i|^p dx = |B_R|, \\
	&\frac{1}{N}|B_R| = \int_{B_R} \frac{|x_i|^2}{|x|^2} dx \le \( \int_{B_R} \frac{|x_i|^p}{|x|^p} dx \)^{\frac{2}{p}} |B_R|^{1-\frac{2}{p}}
\end{align*}
since $p \ge 2$, which implies
\[
	\int_{B_R} \frac{|x_i|^p}{|x|^p} dx \ge \( \frac{1}{N} \)^{\frac{p}{2}} |B_R|.
\]
Thus
\begin{align*}
	\bar{H}_p(B_R) \le \frac{\int_{B_R} |\nabla u_i|^p dx}{\int_{B_R} \frac{|u_i|^p}{|x|^p} dx} \le \frac{|B_R|}{|B_R| (1/N)^{p/2}} =  N^{\frac{p}{2}}.
\end{align*}
Therefore, if $N^{\frac{p}{2}} < \(\frac{N-p}{p}\)^p$, which is equivalent to
\begin{align*}
	N >\(\frac{p + \sqrt{p^2 + 4p}}{2}\)^2,
\end{align*}
the assumption $\bar{H}_p(B_R) <\(\frac{N-p}{p} \)^p$ holds when $p \ge 2$.

When $1 < p < 2$, we take again $u_i(x) = x_i = r \omega_i$, where $r = |x|$ and $\omega_i = \frac{x_i}{r}$.
Then
\begin{align*}
	\int_{B_R} \frac{|u_i|^p}{|x|^p} dx &= \int_0^R \int_{\Sp^{N-1}} \frac{r^p |\omega_i|^p}{r^p} r^{N-1} dr dS_{\omega} = \(\int_0^R r^{N-1} dr \) \(\int_{\Sp^{N-1}} |\omega_i|^p dS_{\omega} \) \\
	&\ge \frac{R^{N}}{N} \int_{\Sp^{N-1}} |\omega_i|^2 dS_{\omega} = \frac{R^{N}}{N} \frac{|\Sp^{N-1}|}{N} 
\end{align*}
since $|\omega_i| \le 1$ and $p \in (1,2)$, where $|\Sp^{N-1}|$ denotes $(N-1)$-dimensional area of the unit sphere $\Sp^{N-1}$.
Thus
\begin{align*}
	\bar{H}_p(B_R) \le \frac{\int_{B_R} |\nabla u_i|^p dx}{\int_{B_R} \frac{|u_i|^p}{|x|^p} dx} \le \frac{|B_R|}{\frac{|\Sp^{N-1}|}{N^2} R^N} 
	= \frac{\frac{|\Sp^{N-1}|}{N} R^N}{\frac{|\Sp^{N-1}|}{N^2} R^N} = N.
\end{align*}
Note that $\lim_{N \to \infty} \frac{\(\frac{N-p}{p}\)^p}{N} = +\infty$ for a fixed $p \in (1,2)$.
Therefore, for given $p \in (1, 2)$, if $N$ is sufficiently large such that $N < \(\frac{N-p}{p}\)^p$,
then we have $\bar{H}_p(B_R) <\(\frac{N-p}{p} \)^p$ for $p \in (1, 2)$.

\vspace{1em}
Similarly, for the assumption of Theorem \ref{Thm2},
let $\Omega = B_R$ be a ball with $0 < R < 1$ and $x = (x_1, \cdots, x_N)$.
Again by testing the Hardy quotient by $u_i(x) = x_i \in X_N$ where $X_N$ is as in \eqref{X_N}, we obtain
\begin{align*}
	\bar{C}_N(B_R) \le \left\{ \frac{\intO |\nabla u_i|^N dx}{\intO \frac{|u_i|^N}{|x|^N (\log \frac{1}{|x|})^N}dx} \right\}
	\le \frac{|B_R|}{\int_0^R \frac{r^{N-1}}{(\log \frac{1}{r})^N} dr \int_{\Sp^{N-1}} |\omega_i|^N dS_{\omega}}.
\end{align*}
Now, we have
\begin{align*}
	&\int_0^R \frac{r^{N-1}}{(\log \frac{1}{r})^N} dr = \int_{N \log \frac{1}{R}}^{\infty} t^{-N} e^{-t} dt \\
	&= \frac{(N \log \frac{1}{R})^{1-N}}{N-1} e^{-N \log \frac{1}{R}} + \frac{1}{1-N} \int_{N \log \frac{1}{R}}^{\infty} t^{1-N} e^{-t} dt,
\end{align*}
and by l'\^Hopital's rule, we see
\begin{align*}
	\lim_{R \to 1-0} \frac{\int_{N \log \frac{1}{R}}^{\infty} t^{1-N} e^{-t} dt}{\int_{N \log \frac{1}{R}}^{\infty} t^{-N} e^{-t} dt}
	= \lim_{R \to 1-0} \frac{t^{1-N} e^{-t}}{t^{-N} e^{-t}} \Big |_{t = N \log \frac{1}{R}} = \lim_{R \to 1-0} N \log \frac{1}{R} = 0. 
\end{align*}
Thus there exists a constant $A = A(N) > 0$ such that 
\[
	\int_0^R \frac{r^{N-1}}{(\log \frac{1}{r})^N} dr = A(N) (\log \frac{1}{R})^{1-N} + o((\log \frac{1}{R})^{1-N}) 
\]
as $R \to 1-0$ and
\begin{align*}
	\bar{C}_N(B_R) &\le \frac{|B_R|}{\( A(N)(\log \frac{1}{R})^{1-N} + o((\log \frac{1}{R})^{1-N})\) \int_{\Sp^{N-1}} |\omega_i|^N dS_{\omega}} \\
	&\le C (\log \frac{1}{R})^{N-1} \to 0, \quad (R \to 1 - 0). 
\end{align*}
Therefore the assumption 
\[
	\bar{C}_N(B_R) < \(\frac{N-1}{N} \)^N
\]
is satisfied for ``fat" balls $B_R$ with $R < 1$ and $1-R$ is sufficiently small.


\section*{Acknowledgment}
The first author (F.T.) was supported by JSPS Grant-in-Aid for Scientific Research (B), No. 23K25781.
This work was partly supported by Osaka City University Advanced Mathematical Institute (MEXT Joint Usage/Research Center on Mathematics and Theoretical Physics). 


\end{document}